\documentclass[12pt]{article}

\usepackage[top=50mm, bottom=50mm, left=50mm, right=50mm]{geometry}
\usepackage{lineno}
\usepackage{amssymb}
\usepackage{amsmath}
\usepackage{amsthm}
\usepackage{bbold}
\usepackage{epsfig}
\usepackage{faktor}
\usepackage{graphicx}
\usepackage{graphics}
\usepackage{pifont}
\usepackage{float}
\usepackage{subfigure}
\usepackage{multirow}
\usepackage{color}
\usepackage{lineno}
\usepackage{fullpage}
\usepackage[normalem]{ulem} 
\usepackage{makeidx}
\usepackage{xspace}
\usepackage{wrapfig}
\usepackage{todonotes}
\usepackage{hyperref}
\usepackage{cleveref}
\usepackage{pst-node}
\usepackage{tikz-cd} 
\usepackage{mathtools}
\makeindex

\newtheorem{Definition}{Definition}[section]

\newtheorem*{theorem*}{Theorem}

\newtheorem*{corollary*}{Corollary}

\newtheorem{Lemma}[Definition]{Lemma}

\definecolor{darkred}{rgb}{1, 0.1, 0.3}
\definecolor{darkblue}{rgb}{0.1, 0.1, 1}
\definecolor{darkgreen}{rgb}{0,0.6,0.5}

\newcommand {\mm}[1] {\ifmmode{#1}\else{\mbox{\(#1\)}}\fi}

\newcommand{\C}{\mathbb{C}} 
\newcommand{\cp}{\mathbb{C}\mathrm{P}} 
\newcommand{\R}{\mathbb{R}} 
\newcommand{\Z}{\mathbb{Z}} 
 
\newcommand{\hz}{\mathrm{HZ}}
\newcommand{\HH}{\mathcal{H}}

\newcommand{\dd}{\mathrm{d}}

\begin{document}

\title{Hofer--Zehnder capacity of magnetic disc tangent bundles over constant curvature surfaces}
 
\author{{Johanna Bimmermann}}

\maketitle

\begin{abstract}
\noindent
We compute the Hofer--Zehnder capacity of magnetic disc tangent bundles over constant curvature surfaces. We use the fact that the magnetic geodesic flow is totally periodic and can be reparametrized to obtain a Hamiltonian circle action. The oscillation of the Hamiltonian generating the circle action immediately yields a lower bound of the Hofer--Zehnder capacity. The upper bound is obtained from Lu's bounds of the Hofer--Zehnder capacity using the theory of pseudo-holomorphic curves. In our case the gradient spheres of the Hamiltonian $H$ will give rise to the non-vanishing Gromov--Witten invariant. 
\end{abstract}

\section{Introduction and main results}
The notion of symplectic capacities was developed to investigate the existence of symplectic embeddings. As symplectomorphisms are always volume preserving one could ask whether a symplectic embedding $M\hookrightarrow N$ exists if and only if there exists a smooth volume preserving embedding. The answer is no (in dimension larger than two) and was given by Gromov in 1985 with his non-squeezing theorem \cite{Gr85}. This means that there must be more global symplectic invariants than volume. A class of such invariants is given by symplectic capacities as introduced by H. Hofer and E. Zehnder in \cite{HZ94}. There, they constructed a special capacity, now known as the Hofer--Zehnder capacity, relating embedding problems with the dynamics on symplectic manifolds. Very importantly, its finiteness implies the existence of periodic orbits on almost all compact regular energy levels (\cite[Ch.\ 4]{HZ94}).
\begin{Definition}
Let $(M,\omega)$ be a symplectic manifold possibly with boundary $\partial M$. We call a smooth Hamiltonian function $H:M\to\R$ \textit{admissible} if there exists a compact subset $K\subset M\setminus \partial M$ and a non-empty open subset $U\subset K$ such that
\begin{itemize}
    \item[a)] 
    $
    H\vert_{M\setminus K}=\max H\ \ \text{and}\ \ H\vert_{U}=0,
    $
    \item[b)] $0\leq H(x)\leq \max H$ for all $x\in M$.
\end{itemize}
Denote by $\mathcal{H}(M)$ the set of admissible functions and by $\mathcal{P}_{\leq 1}(H)$ the set of non-constant periodic orbits with period at most one. The Hofer--Zehnder capacity of the symplectic manifold $(M,\omega)$ is then defined as
$$
c_{\hz}(M,\omega):=\sup\lbrace \max H\ \vert\ H\in\HH(M), \mathcal{P}_{\leq 1}(H)=\emptyset \rbrace.
$$
\end{Definition}
\noindent
In this note we consider the following set up. Let $(\Sigma,g)$ be a closed oriented Riemannian surface of constant curvature $\kappa$. Denote by $\lambda$ the pullback via the metric isomorphism of the canonical 1-form on $T^*\Sigma$. Further denote by $\sigma\in \Omega_2(\Sigma)$ the Riemannian area form. We now study for some $r>0$ the disc tangent bundle
$$
D_r\Sigma=\lbrace (x,v)\in T\Sigma\vert \ g_x(v,v)<r^2\rbrace
$$
and equip it with magnetically twisted symplectic form
$$
\omega_s:=\dd\lambda-s\pi^*\sigma
$$
for some real parameter $s\neq 0$. We call $s$ the strength of the magnetic field. Our main (and only) theorem gives the value of the Hofer--Zehnder capacity for a certain range of $s,r$.
\begin{theorem*}
Let $(\Sigma,g)$ be a closed oriented Riemannian surface of constant curvature $\kappa$. Denote $\sigma$ the corresponding area form and pick two real parameters $r>0$ and $s\neq 0$ satisfying $s^2+\kappa r^2>0$. Then 
$$
c_{\hz}(D_r\Sigma,\dd\lambda-s\pi^*\sigma)=\frac{2\pi}{\kappa}\left (\sqrt{s^2+\kappa r^2}-\vert s\vert\right).
$$
\end{theorem*}
\noindent
The theorem covers three types of surfaces: spheres ($\kappa>0$), flat tori ($\kappa=0$)\footnote{The capacity for $\kappa=0$ is the limit $\lim_{\kappa\to 0}\frac{2\pi}{\kappa}\left (\sqrt{s^2+\kappa r^2}-\vert s\vert\right)=\frac{\pi r^2}{\vert s\vert}$.} and hyperbolic surfaces $(\kappa< 0)$.
The assumption $s^2+\kappa r^2>0$ does not put any additional constraint on spheres and flat tori. For hyperbolic surfaces it tells us to look at strong magnetic fields, i.e. $\vert s\vert >\sqrt{-\kappa}\lambda$.\\
\ \\
\noindent
There are some modifications of the Hofer--Zehnder capacity we will also discuss. First one can look at this capacity with respect to a fixed free homotopy class of loops $\nu$. We denote
$$
\mathcal{P}_T(H;\nu):=\lbrace \gamma\in C^\infty(\R/T\Z, M)\ \vert \ \dot\gamma(t)=X_H(\gamma(t))\neq 0;\ [\gamma]=\nu\rbrace
$$
the set of non-constant $T$-periodic solutions to the Hamiltonian equations in the class $\nu$ and by $\mathcal{P}_{\leq T}(M,\nu)$ the set of non-constant periodic solutions in class $\nu$ with period less or equal to $T$. The Hofer--Zehnder capacity with respect to this free homotopy class is defined to be
$$
c_{\hz}^\nu(M,\omega):=\sup\lbrace \max H\ \vert\ H\in \HH(M), \mathcal{P}_{\leq 1}(H;\nu)=\emptyset\rbrace.
$$
Second there is a relative version, considering only Hamiltonians constant along a subset $Z\subset M$ not touching the boundary. It is defined as follows. Denote by $\HH(M,Z)$ the set of smooth functions satisfying 
\begin{itemize}
    \item[a)] 
    $
    H\vert_{M\setminus K}=\max H\ \ \text{and}\ \ H\vert_{U}=0,
    $
    \item[b)] $0\leq H(x)\leq \max H$ for all $x\in M$,
\end{itemize}
for an open neighborhood $U\supset Z$ and a compact set $K\supset U$. The relative Hofer--Zehnder capacity is then defined as 
$$
c_{\mathrm{HZ}}(M, Z, \omega):=\sup\lbrace \max H\ \vert\ H\in\HH(M,Z), \mathcal{P}_{\leq 1}(H)=\emptyset\rbrace.
$$
Observe that from the definitions it follows directly that 
$$
c_{\hz}(M,Z,\omega)\leq c_{\hz}(M,\omega)\leq c_{\hz}^0(M,\omega).
$$
Proving the theorem it becomes clear that the Hamiltonian we use for the lower bound is actually constant along the zero-section, thus also bounds the relative capacity from below. In addition the upper bound comes from pseudo-holomorphic spheres, therefore it detects contractible orbits which means we actually bound the $\pi_1$-sensitive capacity from above. Thus the following corollary follows.
\begin{corollary*} The $\pi_1$-sensitive and the relative Hofer--Zehnder capacity agree, i.e.
$$
    c_{\hz}(D_r\Sigma,\Sigma,\omega_s)=c_{\hz}(D_r\Sigma,\omega_s)=c_{\hz}^0(D_r\Sigma,\omega_s).
$$    
\end{corollary*}
\noindent
This example also shows that the $\pi_1$-sensitive capacity is not continuous on all smooth families of domains bounded by smooth hypersurfaces, a question raised by Cieliebak, Hofer, Latschev and Schlenk in \cite[Prob.\ 7]{CHLS}. Indeed for closed hyperbolic surfaces $(\kappa=-1)$ we find that
$$
c_{\hz}^0(D_r\Sigma,\omega_1)=\left\{\begin{array}{ll} 2\pi(1-\sqrt{1-r^2})\ & \text{for}\quad   r\leq 1\\
         \infty \  &\text{for}\quad  r>1 \end{array}\right.
$$
is not continuous in $r$. It jumps precisely at the Ma\~{n}\'{e} critical value \cite[Sec.\ 5.2]{CFP10}. To see that the $\pi_1$-sensitive capacity for $r>1$ is infinite, observe that any radial Hamiltonian that is constantly zero on $D_1\Sigma$ has no contractible periodic orbits, as curves in $\C\mathrm{H}^1$ of constant geodesic curvature less than 1 are not periodic.
\\
\ \\
\noindent
\textbf{Outline.}\\
In the second section we modify the magnetic geodesic flow to obtain a semifree Hamiltonian circle action and in the third section we use this circle action to prove the theorem and its corollary.\\
\ \\
\noindent
\textbf{Acknowledgments.}\\
I want to thank Gabriele Benedetti and Kai Zehmisch for their valuable feedback and input on the topic. This work is part of a project in the SFB/TRR 191
{\it Symplectic Structures in Geometry, Algebra and Dynamics},
funded by the DFG. The author acknowledges funding by the Deutsche Forschungsgemeinschaft (DFG, German Research Foundation) – 281869850 (RTG 2229), 390900948 (EXC-2181/1) and 281071066 (TRR 191).

\section{Magnetic geodesic flow}
The circle action we are going to construct will be a reparametrization of the magnetic geodesic flow. Lets first work on the universal cover of $\Sigma$, hence $\cp^1,\C^1,\C\mathrm{H}^1$ depending on the sign of the curvature. As for example shown in \cite{BR19} magnetic geodesics on on these spaces are curves of constant geodesic curvature $\kappa_g=\vert \frac{s}{v}\vert$. If $R$ denotes the radius (with respect to the Riemannian metric $g$) of a geodesic circle we know, using normal polar coordinates, that its circumference $C$ and the geodesic curvature $\kappa_g$ are

\begin{align}\label{e5}
   C&=\frac{2\pi}{\sqrt{\kappa}}\sin(\sqrt{\kappa}R)=\frac{2\pi\sqrt{\kappa}^{-1}\tan(\sqrt{\kappa}R)}{\sqrt{1+(\tan(\sqrt{\kappa}R))^2}},\\ 
   \kappa_g&=\frac{\sqrt{\kappa}}{\tan(\sqrt{\kappa}R)}.\label{e4}
\end{align}
Here we use the convention $\sqrt{-1}=i$ and the formulas $-i\sin(ix)=\sinh(x)$, $\cos(ix)=\cosh(x)$. Observe that in the hyperbolic case the geodesic curvature of geodesic circles can not be less than $\sqrt{\vert\kappa\vert}$. Indeed curves of geodesic curvature less than $\sqrt{\vert\kappa\vert}$ do not close up.
We therefore restrict to the regime of strong magnetic field, i.e.\ $s^2+\kappa \vert v\vert^2>0$. See Figure \ref{fig8} for a visualisation.
\begin{figure}[]
	\centering
 \includegraphics[width=1\textwidth]{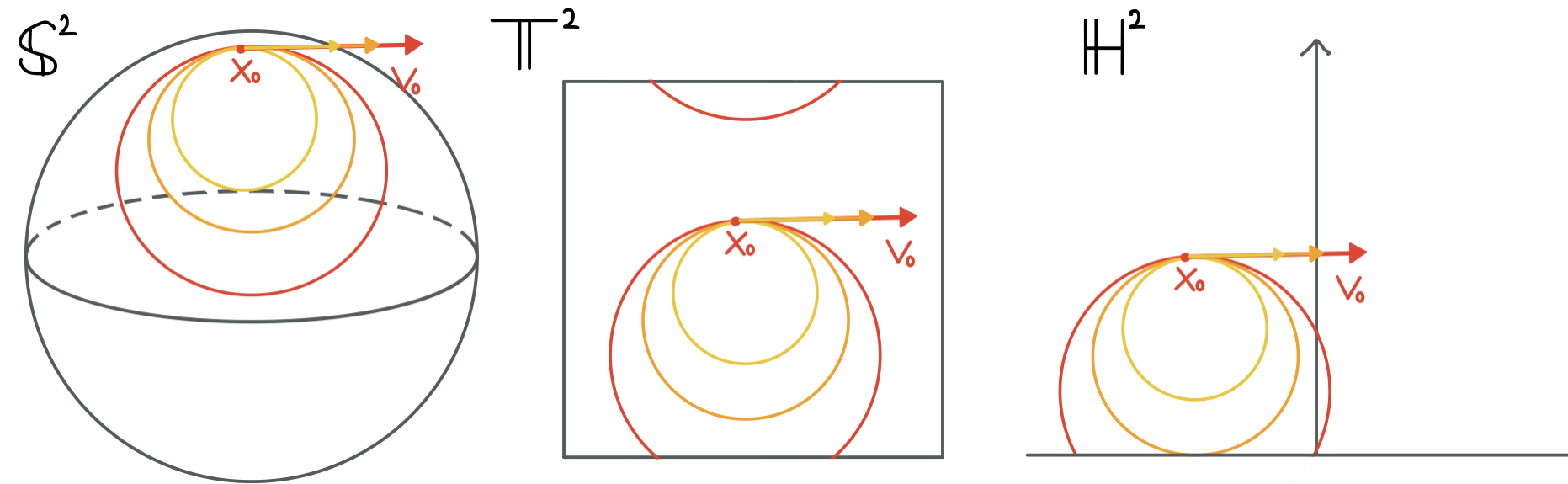}
	\caption{\textit{The picture shows families of geodesic circles. The vectors $v_0$ of different length indicate that the corresponding magnetic geodesic is a parametrized geodesic circle of geodesic curvature $\kappa_g=s/\vert v_0\vert$.}}
    \label{fig8}
\end{figure}
\noindent
Inserting $\kappa_g$ into \eqref{e5} yields
$$
C=\frac{2\pi}{\kappa_g\sqrt{1+\kappa/\kappa_g^2}}=\frac{2\pi\vert v\vert}{\sqrt{s^2+\kappa\vert v\vert^2}},
$$
where in the last step we inserted $\kappa_g=s/\vert v\vert$. Now, we conclude that the period is given by
$$
T=\frac{C}{\vert v\vert}=\frac{2\pi}{\sqrt{s^2+\kappa\vert v\vert^2}}.
$$
In particular the reparametrization $H=h\circ E$ with 
$$
h:\R_{\geq 0}\to\R;\ \  h(E)= \frac{2\pi}{\kappa}\left( \sqrt{s^2+2\kappa E}-\vert s\vert \right)
$$
induces a Hamiltonian $S^1$-action (of period $T=1)$. Observe that the induced circle action is semifree as all, but the constant orbits, have period 1. If we now consider arbitrary Riemannian surfaces of constant curvature, it is not clear that the induced circle action is still semifree. This is the statement of the following proposition.
\begin{Lemma}\label{thm4}
Let $(\Sigma,g,j)$ be an orientable Riemann surface of constant sectional curvature $\kappa$. Then for constants $s\in\R\setminus\lbrace 0\rbrace$ and $ r>0$ satisfying $s^2+\kappa r^2>0$ the Hamiltonian 
$$
H: D_r M\to \R;\ (x,v)\mapsto \frac{2\pi}{\kappa}\left(\sqrt{s^2+\kappa\vert v\vert^2}-\vert s\vert\right)
$$
generates a semifree Hamiltonian circle action on the disc-subbundle $(D_r M, \omega_s)$ of the magnetically twisted tangent bundle. 
\end{Lemma}
\begin{proof} If $\Sigma$ is simply connected we are done by the previous computations. If $\Sigma$ is not simply connected we know that $\kappa\leq 0$ and $\Sigma=\Tilde{\Sigma}/\Gamma$ for a discrete subgroup $\Gamma$ of isometries acting freely on the universal covering $\Tilde{\Sigma}\in\lbrace \C^1,\C\mathrm{H}^1\rbrace$. We need to make sure that the restriction of the projection $\Tilde{\Sigma}\to \Sigma$ to any magnetic geodesic $\Tilde{\gamma}\to \gamma$ is no covering of degree $>1$. The prove goes by contradiction. Assume $\Tilde{\gamma}$ was covering $\gamma$ with some degree $>1$, then there must be an element $g\in \Gamma$ that is a rotation around the center of $\Tilde{\gamma}$. In particular $g$ fixes the center of $\Tilde{\gamma}$, which yields a contradiction as $\Gamma$ acts freely on $\Tilde{\Sigma}$.
\end{proof}

\section{Proof of Theorem}
In the previous section we proved that for $r>0$, satisfying $s^2+\kappa r^2>0$, the Hamiltonian
$$
H:D_{r}\Sigma\to \R; \ H(x,v)=\frac{2\pi}{\kappa}\left(\sqrt{s^2+\kappa \vert v\vert^2}-\vert s\vert\right)
$$
is well-defined, smooth and generates a semi-free $S^1$-action. We can now show that the oscillation of this Hamiltonian yields both a lower and an upper bound for the Hofer--Zehnder capacity and thus determines it.\\
\ \\
\noindent
\textbf{Lower bound:}\\
We modify the Hamiltonian $H$ generating the circle action slightly so that it becomes admissible. This can be done with the help of a function
    $f:[0,\max H]\to [0,\infty)$ satisfying
    $$
    \begin{aligned}
    &0\leq f'(x)< 1, \\
    &f(x)=0\ \  \text{near}\ \ 0,\\
    &f(x)=\max H-\varepsilon\ \  \text{near}\ \ \max H.
    \end{aligned}
    $$
    Then all solutions to the Hamiltonian system with Hamiltonian $\tilde H=f\circ H$ have period 
    $$
    T=\frac{1}{f'(E)}> 1.
    $$
    Thus $\tilde H$ is admissible and we find the estimate 
    $$
    c_{\hz}(D_rM,\omega_s)\geq \max(\tilde H)=\max(H)-\varepsilon,\ \ \ \ \forall \varepsilon>0.
    $$
\noindent
\textbf{Upper bound:}\\
We do a Lerman-cut (\cite{ler}) at the regular energy surface $\lbrace \vert v\vert=r\rbrace$ to compactify $(D_r\Sigma, \omega_s)$. The compactification $(\overline{D_r\Sigma, \omega_s)}$ is a closed symplectic 4-manifold with semi-free Hamiltonian circle action. The critical set, where the Hamiltonian attains its minimum, corresponds to the zero-section and is thus of codimension two. By \cite[Prop.\ 4.3]{Duff09} the 1-point Gromov-Witten invariant $\mathrm{GW}_A([pt.])$ in the class $A\in H_2(\overline{D_r\Sigma},\Z)$ of a gradient sphere $u(s,t)=\Phi^t_H\gamma(s)$ ($A=[u]$) is non-vanishing\footnote{Here $\gamma$ denotes a gradient flow line with respect to some compatible metric.}. Further the homology class represented by the divisor $D_\infty$ obtained from collapsing the boundary is proportional to the Poincaré-dual of $[\bar\omega]$. Thus 
$$
\mathrm{GW}_A([pt.],[D_\infty],[D_\infty])=\mathrm{GW}_A([pt.])\left (A\cdot[D_\infty]\right)^2\neq 0.
$$
This means we can apply Lu's theorem \cite[Thm.\ 1.10]{Lu06} to obtain as upper bound
\begin{align*}
    c_{\hz}(D_r\Sigma,\dd\lambda-s\pi^*\sigma)&=c_{\hz}(\overline{D_r\Sigma}\setminus D_\infty,\bar\omega)\\
    &\leq \omega(A)=\int_{\cp^1} u^*\omega=\int_{-\infty}^\infty \dd s\int_0^1\dd t\ \omega(\partial_s u,\partial_tu)\\
    &=\int_{-\infty}^\infty \dd s\int_0^1\dd t\ \dd H(\dot\gamma(t))= \frac{2\pi}{\kappa}\left(\sqrt{s^2+\kappa r^2}-\vert s\vert\right).
\end{align*}
\ \\
\noindent
\textbf{Other types of the Hofer--Zehnder capacity:}\\
    Observe that Lu's theorem actually yields an upper bound for the $\pi_1$-sensitive Hofer--Zehnder capacity. Further the Hamiltonian $\tilde H$ used for the lower bound is constant along the zero-section, thus $\max\tilde H$ also bounds the relative Hofer-Zehnder from below. In total we obtain
    $$
    c_{\hz}(D_r\Sigma,\Sigma,\omega_s)=c_{\hz}(D_r\Sigma,\omega_s)=c_{\hz}^0(D_r\Sigma,\omega_s).
    $$
    As all orbits of $\tilde H$ are contractible we can further conclude that 
    $$
    c_{\hz}^\nu(D_r\Sigma,\omega_s)=\infty
    $$ for any $\nu\neq 0$. 
    All these statements hold only under the condition $s\neq 0$ and $s^2+\kappa r^2>0$. The cases not covered by this condition are the sphere ($\kappa>0$) and tori ($\kappa=0$) with vanishing magnetic field (s=0) and hyperbolic surfaces ($\kappa<0$) with weak magnetic field. As shown in \cite{B23} the theorem and the corollary actually hold in the case of spheres with vanishing magnetic field. In the other cases it is only clear that 
    $$
    c_{\hz}^0(D_r\Sigma,\omega_s)=\infty
    $$
    as the magnetic geodesic flow has no contractible periodic orbits on high energy levels. Further for tori finiteness of the full Hofer--Zehnder capacity is known (\cite[Prop.4, Ch.4]{HZ94}), but for hyperbolic surfaces even finiteness is unknown.
\newpage
\bibliographystyle{abbrv}
\bibliography{ref}

\end{document}